\documentclass[a4paper,10pt]{scrartcl}

\usepackage[fleqn]{amsmath}
\usepackage{amssymb}
\usepackage{amsthm}
\usepackage{stmaryrd}
\usepackage{bbm}
\usepackage{hyperref}
\usepackage{mathtools}
\usepackage[left=3cm,right=3cm,top=3cm,bottom=3cm]{geometry}
\usepackage{qtree}
\usepackage{graphicx}
\usepackage{wrapfig}

\mathtoolsset{showonlyrefs}

\DeclareMathOperator{\Id}{Id}

\DeclareMathOperator{\Var}{Var}

\DeclareMathOperator{\supp}{supp}
\DeclareMathOperator{\Ent}{Ent}

\newcommand{\bE}{\ensuremath{\mathbb{E}}}

\newcommand{\bN}{\ensuremath{\mathbb{N}}}

\newcommand{\bP}{\ensuremath{\mathbb{P}}}

\newcommand{\bR}{\ensuremath{\mathbb{R}}}

\newcommand{\bZ}{\ensuremath{\mathbb{Z}}}

\newcommand{\ind}{\ensuremath{\mathbbm{1}}}

\newcommand{\cE}{\ensuremath{\mathcal{E}}}

\newcommand{\cG}{\ensuremath{\mathcal{G}}}

\newcommand{\cN}{\ensuremath{\mathcal{N}}}

\newcommand{\cS}{\ensuremath{\mathcal{S}}}
\newcommand{\cT}{\ensuremath{\mathcal{T}}}

\newcommand{\cX}{\ensuremath{\mathcal{X}}}

\newcommand{\bs}{\backslash}

\newcommand{\abs}[1]{\left\vert \, #1 \, \right\vert}
\newcommand{\norm}[1]{\left\Vert \, #1 \, \right\Vert}

\newcommand{\normb}[1]{\interleave \, #1 \, \interleave}

\newcommand{\ddx}[1][1]{\ifnum#1=1 \frac{d}{dx} \else \frac{d^{#1}}{dx^{#1}} \fi}
\newcommand{\ddy}[1][1]{\ifnum#1=1 \frac{d}{dy} \else \frac{d^{#1}}{dy^{#1}} \fi}
\newcommand{\ddt}[1][1]{\ifnum#1=1 \frac{d}{dt} \else \frac{d^{#1}}{dt^{#1}} \fi}

\newcommand{\intc}[2]{\int_{\mathrlap{#1}}^{\mathrlap{#2}}\;}

\newtheorem{theorem}{Theorem}[section]
\newtheorem{lemma}[theorem]{Lemma}

\newtheorem{corollary}[theorem]{Corollary}

\newenvironment{remark}[1][Remark]{\begin{trivlist}
\item[\hskip \labelsep {\bfseries #1}]}{\end{trivlist}}

\newcommand{\oT}{{\mathring{T}}}
\newcommand{\PsiSX}[1][X]{{\Psi_{S,#1}}}
\newcommand{\oN}{{\overline{N}}}
\newcommand{\oS}{{\overline{S}}}
\newcommand{\uS}{{\underline{S}}}
\newcommand{\oNt}{{\overline{N}_{[0,t]}}}

\newcommand{\oNSv}{{\overline{N}_{[\uS(v),\oS(v)[}}}

\newcommand{\hE}{\widehat{\bE}}
\newcommand{\hP}{\widehat{\bP}}

\begin{document}

\title{Talagrand's inequality for Interacting Particle Systems satisfying a log-Sobolev inequality}
\author{Florian V\"ollering\thanks{University of G\"ottingen, IMS, Goldschmidtstra\ss e 7
37077 G\"ottingen,
Germany \newline email: florian.voellering@mathematik.uni-goettingen.de, tel.: +49551 3913520}}

\maketitle

\begin{abstract}
Talagrand's inequality for independent Bernoulli random variables is extended to many interacting particle systems (IPS). The main assumption is that the IPS satisfies a log-Sobolev inequality. In this context it is also shown that a slightly stronger version of Talagrand's inequality is equivalent to a log-Sobolev inequality.

Additionally we also look at a common application, the relation between the probability of increasing events and the influences on that event by changing a single spin. 
\end{abstract}

\hspace{1cm}\newline

\section{Introduction}
In the famous paper \cite{TALAGRAND:94}, Talagrand proved that for Bernoulli measures $\mu_p$, $0<p<1$, on the discrete 
hypercube $\{0,1\}^N$ the following inequality holds:
\begin{align}\label{eq:talagrand-original}
 \Var_\mu(f) \leq C \sum_{i=1}^N \frac{\norm{\nabla_i f}_{L^2(\mu_p)}^2}{1+\log\left(\norm{\nabla_i f}_{L^2(\mu_p)}/\norm{\nabla_i f}_{L^1(\mu_p)}\right)}.
\end{align}
Here $f$ is an arbitrary function in $L^2(\mu_p)$, $\nabla_i f (\eta) = f(\eta^i)-f(\eta)$, $\eta^i$ is obtained from $\eta$ by changing the $i$th coordinate, and $C=C(p)$ is a constant dependent only on $p$. 

This inequality has found many applications, for example in random graph theory \cite{FRIEDGUT:99} and percolation \cite{BENJAMINI:KALAI:SCHRAMM:03}, and has as a consequence the widely used KKL-bound for influences of Boolean functions \cite{KAHN:KALAI:LINIAL:88}. However the original proof is limited to product measures on the discrete hypercube. In \cite{VANDENBERG:KISS:12} this restriction to independent coordinates was circumvented by encoding dependent random variables as functions of a larger set of independent random variables. It was shown that if the encoding is in some sense not too far from independent random variables, then an analogue to \eqref{eq:talagrand-original} holds. 
There are also versions of Talagrand's inequality and related inequalities like the KKL bound on continuous spaces \cite{BOURGAIN:KAHN:KALAI:92} \cite{CORDERO-ERAUSQUIN:LEDOUX:12} \cite{KELLER:MOSSEL:SEN:12}.
In \cite{CORDERO-ERAUSQUIN:LEDOUX:12} Talagrand's inequality is proven under the assumption that there are Markovian dynamics which satisfy the log-Sobolev inequality and a specific permutation relation between directional derivatives and the dynamics. While this approach is fruitful for products of Gaussian measures or the 
uniform measure on the sphere the required permutation relation is too restrictive for typical interacting particle systems. 

In this paper we extend Talagrand's inequality to a wide class of interacting particle systems under the assumption of a log-Sobolev inequality. For example many finite range Gibbs measures are included. An important application of Talagrand's inequality is the study of influences on increasing events. By proving a version of Russo's formula for dependent random variables we can obtain similar estimates to the i.i.d. case. 

We will prove under some conditions an equivalence between a version of Talagrand's inequality and a log-Sobolev inequality. Harder to prove, and from the point of view of applications more important, is the implication from a log-Sobolev inequality to Talagrand's inequality. Sufficient conditions for a log-Sobolev inequality are available in the context of interacting particle systems, and Talagrand's inequality as a consequence adds a useful variance estimate for the study of those.


\section{Setting and Results}
\subsection{Notation and Definitions}
Let $E$ be a Polish space, $\cG$ a finite or countable group and $\Omega:=E^{\cG}$. The most common choice is $E=\{0,1\}$, $\cG=\bZ^d$. Elements of $\Omega$ we will denote by $\eta,\xi$, and elements of $\cG$ by $x,y,z$. 
Given a probability measure $\mu$ on $\Omega$ and $p\geq 1$, we write
\[ \norm{f}_p = \norm{f}_{L^p(\mu)} = \left(\int \abs{f}^p \,d\mu\right)^{\frac1p} .\]
By $\Phi$ we denote the convex function $\Phi(x)=\frac{x^2}{\log(e+\abs{x})}$. To $\Phi$ we consider the associated Orlicz norm,
\begin{align}\label{eq:Orlicz-norm}
 \norm{f}_\Phi := \inf\left\{ a>0 : \int \Phi\left(\frac{f}{a}\right)\,d\mu\leq 1 \right\}.
\end{align}
Let $L$ be the generator of a Markov process on $\Omega$, which we assume to be reversible with respect to an invariant and ergodic probability measure $\mu$. The Dirichlet from associated to $(L,\mu)$ is given by $\cE(f,g)=-\int f Lg\,d\mu$. The pair $(L,\mu)$ is said to satisfy a Poincar\'e inequality if there is a constant $\kappa>0$ so that
\begin{align}\label{eq:Poincare}
 \Var_\mu(f) = \int \left(f-\int f d\mu\right)^2 d\mu \leq \frac1\kappa \cE(f,f), \quad \forall f\in L^2(\mu).
\end{align}
Similarly, $(L,\mu)$ satisfies a logarithmic Sobolev inequality if there is a constant $\rho>0$ so that
\begin{align}\label{eq:log-Sobolev}
 \Ent_\mu(f^2) = \int f^2 \log\left(\frac{f^2}{\norm{f}_2^2}\right)\,d\mu \leq \frac2\rho\cE(f,f), \quad \forall f\in L^2(\mu).
\end{align}
The best constants $\kappa,\rho$ in \eqref{eq:Poincare} and \eqref{eq:log-Sobolev} are called the spectral gap and the logarithmic Sobolev constant, which we will also call $\kappa$ and $\rho$. It is well known that $\rho\leq\kappa$.

Let $(P_t)_{t\geq0}$ be the semi-group generated by $L$. Both \eqref{eq:Poincare} and \eqref{eq:log-Sobolev} imply contraction properties, \eqref{eq:L2-decay} and \eqref{eq:hypercontractivity}, respectively:
\begin{align}
 \norm{P_t f}_2 &\leq e^{-\kappa t} \norm{f}_2;	\label{eq:L2-decay}\\
 \norm{P_t f}_q &\leq \norm{f}_{p(t,q)},\quad p(t,q) = 1+(q-1)e^{-2\rho t}.	\label{eq:hypercontractivity}
\end{align}
The reader unfamiliar with these notions may be referred to \cite{GUIONNET:ZEGARLINSKI:03}. In the context of interacting particle systems and Gibbs measures there are many interrelated notions of functional inequalities and mixing conditions on finite and infinite volume, which are also connected to decay of the semi-group. See \cite{HOLLEY:STROOK:87},\cite{MARTINELLI:OLIVIERI:94a}, \cite{MARTINELLI:OLIVIERI:94} for a study of those. 

We want to exploit some of the geometry available on $\Omega$. To do so we assume that the generator is of the form
\begin{align}\label{eq:generator}
 Lf(\eta) := \sum_{x\in\cG} \int_E \left(f(\eta_{x\mapsto a})-f(\eta)\right) \,\mu_{x,\eta}(d a),
\end{align}
where $(\mu_{x,\eta})_{\eta\in\Omega,x\in\cG}$ is a collection of probability measures on $E$, and $\eta_{x\mapsto a}$ is the configuration which is identical to $\eta$ except at $x$, where it has value $a$. To avoid notation we will also write $\mu_{x,\eta}(d\xi)$ with the understanding that $\xi=\eta_{x\mapsto a}$. We assume that the interaction is \emph{finite range}: There exists a finite $\cN\subset \cG$ so that if $\eta'(y)=\eta(y)$ for all $y\in x+\cN$, then $\mu_{x,\eta'}=\mu_{x,\eta}$. Glauber dynamics of interacting particle systems are a standard example for a generator of the form \eqref{eq:generator}.

Based on the generator $L$ and the measures $\mu_{x,\eta}$ we define for $x\in\cG$ the linear operator $\Psi_x$ on $L^p(\mu),1\leq p\leq \infty$, via
\begin{align}\label{eq:Psix}
 \Psi_x f (\eta) := \int f(\xi) \,\mu_{x,\eta}(d\xi).
\end{align}
It is a consequence of the reversibility of $L$ that
\begin{align}
 \int \Psi_x f \,d\mu = \int f \,d\mu. \label{eq:Psi-invariant}
\end{align}
From \eqref{eq:Psi-invariant} follows an alternative way of writing the Dirichlet form:
\begin{align}
 \cE(f,f) &= \frac12\sum_{x\in\cG}\int \left[\Psi_x(f-f(\eta))^2\right](\eta)\,\mu(d\eta).
\end{align}
Define the derivative of a function $f$ in direction $x$ as
\begin{align}\label{eq:Dx}
 D_x f := \Psi_x f - f.
\end{align}
Note that this derivative is derived from the jump dynamics and does not necessarily agree with a possible natural derivative on $E$. 

We will need that the derivatives of a function and the Dirichlet form are compatible in a certain sense. To make this concrete, a function $f\in L^2(\mu)$ is called \emph{good with constant $K<\infty$}, if 
\begin{align}\label{eq:reverse-Jensen}
\cE(P_tf,P_tf)\leq K\sum_{x\in\cG}\norm{D_x P_tf}_2^2, \quad \forall\; t\geq0.
\end{align}
To put \eqref{eq:reverse-Jensen} into perspective, observe that by Jensen's inequality we have $\sum_{x\in\cG}\norm{D_x f}_2^2\leq 2\cE(f,f)$. In this sense \eqref{eq:reverse-Jensen} is a form of a reverse Jensen's inequality, which may seem restrictive. However, under fairly mild conditions we have in fact that all functions in $L^2(\mu)$ are good.
\begin{lemma}\label{lemma:IPS-K-condition}
Suppose $E$ is finite and $\alpha:=\inf\{\mu_{x,\eta}(e) : \eta\in\Omega,x\in\cG,e\in E\}>0$. Then any $f\in L^2(\mu)$ is good with constant $\alpha^{-3}$.
\end{lemma}
\subsection{Talagrand's inequality}
\begin{theorem}\label{thm:talagrand}
Assume that $(L,\mu)$ satisfies the log-Sobolev inequality \eqref{eq:log-Sobolev} with constant $\rho>0$, that $L$ is of the form \eqref{eq:generator} and the interaction is finite range. Let $f\in L^2(\mu)$ be a good function with constant $K<\infty$. Then there exists a constant $C>0$ so that
\begin{align}\label{eq:talagrand-orlicz}
 \Var_\mu(f) \leq C \sum_{x\in\cG} \norm{D_x f}^2_\Phi.
\end{align}
The constant $C$ is depends only on $K$, $\abs{\cN}$ and $\rho$. 
\end{theorem}
Talagrand's inequality of the form \eqref{eq:talagrand-original} is a direct consequence from \eqref{eq:talagrand-orlicz}, as was already noted in \cite{TALAGRAND:94}.
\begin{corollary}\label{cor:talagrand}
Under the assumptions of Theorem \ref{thm:talagrand},
 \begin{align}\label{eq:talagrand-corollary}
  \Var_\mu(f) \leq c_1 C \sum_{x\in\cG}\frac{\norm{D_x f}_2^2}{1+\log\left(\norm{D_x f}_2/\norm{D_x f}_1\right)},  
 \end{align}
 where $c_1$ is a universal constant and terms on the right hand side where $D_x f$ is almost surely constant are treated as 0.
\end{corollary}
\begin{proof}
 The claim follows directly by virtue of Lemma 2.5 in \cite{TALAGRAND:94}, which states that
there is a universal constant $C'$ so that
 \[ \norm{f}_\Phi^2 \leq C' \frac{\norm{f}^2_2}{1+\log\left(\frac{\norm{f}_2}{\norm{f}_1}\right)} . \qedhere\]
\end{proof}
It is also possible to generalize Theorem \ref{thm:talagrand} to product spaces.
\begin{theorem}\label{thm:tensor-talagrand}
 Suppose we have, for $i=1,...,N$, the spaces $\Omega_i:=E_i^{\cG_i}$ with probability measures $\mu_i$ and dynamics generated by $L_i$ satisfying the conditions of Theorem \ref{thm:talagrand}. Let $\mu$ be the product measure of the $\mu_i$ on $\Omega=\Omega_1\times...\times\Omega_N$. Let $L=L_1+...+L_N$ be the generator of the product semigroup $P_t$. Let $f\in L^2(\mu)$ be a good function with constant $K<\infty$ (where $\sum_{x\in\cG}$ is replaced by $\sum_{i=1}^N\sum_{x\in\cG_i}$ in the definition of a good function). Then there exists a constant $C>0$ so that
 \begin{align}
 \Var_\mu(f) \leq C\sum_{i=1}^N\sum_{x\in\cG_i}\norm{D_x f}^2_\Phi.
 \end{align}
 The constant $C$ is of the same form as in Theorem \ref{thm:talagrand}, but with $\rho$ replaced by the minimum over the $\rho_i$ and $\abs{\cN}$ replaced by the maximum over the $\abs{\cN_i}$.
\end{theorem}

If one compares \eqref{eq:Poincare}, \eqref{eq:log-Sobolev}, \eqref{eq:talagrand-orlicz} and \eqref{eq:talagrand-corollary} one observes that we have the following chain of implications:
\[ \eqref{eq:log-Sobolev} \Rightarrow \eqref{eq:talagrand-orlicz} \Rightarrow \eqref{eq:talagrand-corollary}  \Rightarrow \eqref{eq:Poincare}. \]
It turns out that the first implication is in fact an equivalence. That is, not only does a log-Sobolev inequality  imply Talagrand's inequality, but the converse is also true when using the version with the Orlicz norm.
\begin{theorem}\label{thm:reverse-talagrand}
Let $L$ be of the form \eqref{eq:generator}.
 Suppose that a function $f:\Omega\to\bR$ satisfies 
 \begin{align}\label{eq:reverse-talagrand-assumption}
 \Var_\mu(f) \leq C \sum_{x\in\cG} \norm{D_x f}^2_\Phi
 \end{align}
 for some $C\geq 1$. Then
 \[ \Ent_\mu(f) \leq c_2 C \cE(f,f), \]
 where $c_2$ is a universal constant.
\end{theorem}
\begin{corollary}
 Let $L$ be of the form \eqref{eq:generator} with finite range. Assume that all $f\in L^2(\mu)$ are good with the same constant $K<\infty$. Then the following two statements are equivalent:
 \begin{enumerate}
  \item $(L,\mu)$ satisfies a log-Sobolev inequality with constant $\rho>0$;
  \item there is a constant $C>0$ so that \eqref{eq:talagrand-orlicz} is satisfied for all $f\in L^2(\mu)$.
 \end{enumerate}
\end{corollary}

In \cite{CORDERO-ERAUSQUIN:LEDOUX:12} the proof of Talagrand's inequality consists of two key ingredients, namely hypercontractivity and a permutation relation between $D_x$ and $P_t$. We will also use these two ingredients. However, the original permutation relation,
\begin{align}\label{eq:commutation}
D_x P_t f \leq e^{Kt}P_t D_x f,\qquad 0\leq t\leq t_0, 
\end{align}
cannot hold in the context of interacting particle systems, as the following argument shows. Let $E=\{0,1\}$, and $f(\eta)=\eta(x)$ for some $x\in\cG$. Then $D_y f = 0$ for all $y\neq x$. But $D_y P_t f \neq 0$ for typical dynamics (with independent spin flips being an exception). Hence we need an alternative to \eqref{eq:commutation} which respects the space-time structure generated by the Markov process. Most of the paper is devoted  prove the following commutation property between the semi-group $P_t$ and the derivative operator $D_x$: 
\begin{theorem}\label{thm:nice}
 Assume that the generator $L$ is of the form \eqref{eq:generator} and the interaction is finite range. Then
\begin{align}\label{eq:nice-thm}
 \sum_{x\in\cG} \norm{D_x P_t f}^2_2 \leq \widetilde C 2^{ t}\sum_{x\in\cG}\norm{D_x f}^2_{p(t)}, 
\end{align}
with $\widetilde C = 2e^{72 \abs{\cN}^2(1+\abs{\cN})^2}$ and $p(t)=p(t,2) = 1+e^{-2\rho t}$. The constant $\rho$ is again the log-Sobolev constant. However $\rho=0$ is admitted as well, in which case we have the $L^2$-norm on both sides of \eqref{eq:nice-thm}. 
\end{theorem}

The proof of Theorem \ref{thm:nice} is at its core a recursive strategy. We will show that $\sum_{x\in\cG} \norm{D_x P_t f}^2_2$ can be estimated against $\sum_{x\in\cG}\norm{D_x f}^2_{p(t)}$ plus an error term of order 1. This error term can be estimated against $\sum_{x\in\cG}\norm{D_x f}^2_{p(t)}$ as well, but in doing so we introduce two error terms of order 2. This will then be iterated. The entire procedure is rather technical and is therefore done in the last sections of the paper. 

\subsection{On Influences and Russo's formula} \label{section:influences}
An important application of Talagrand's inequality is the study of the sensitivity of events to changes at a single site. We can obtain the same type of estimates as in \cite{TALAGRAND:94} even in the context of dependent random variables. To this end we need to generalize is Russo's formula. 

In this section we restrict ourselves to $E=\{0,1\}$. We say a subset $A\subset\Omega=E^\cG$ is increasing, if $\eta\in A$ and $\xi\geq\eta$ (coordinate wise) implies $\xi\in A$. Denote by $A_x$ the event $\{\eta\in A, \eta^x\not\in A\}$, where $\eta^x$ is the configuration $\eta$ flipped at $x$, that is a 1 at $x$ is replaced by a 0 and vice versa.

Let $\mu_p$, $p\in[a,b]$, be a family of measures on $\Omega$. We assume that the corresponding dynamics given by the generator $L_p$ from \eqref{eq:generator} are the heat bath Glauber dynamics. That is,
\begin{align}\label{eq:mu-is-glauber}
\mu^p_{x,\xi}(1) = \mu_p\left( \eta(x)=1 \;\middle|\;\eta=\xi\text{ off }x\right).
\end{align}
We also assume that the map $p\mapsto\mu_{x,\eta}^p(1)$ is increasing and differentiable for all $x\in\cG,\eta\in\Omega$, and that all $\mu_{x,\eta}^p$ are finite range with the same neighborhood $\cN$.

\begin{theorem}[Russo's formula for dependent random variables]\label{thm:Russo}
Write $\beta_p:=\inf_{\eta\in\Omega,x\in\cG}\frac{d}{dp}\mu^p_{x,\eta}(1)$. Then, for any $p\in]a,b[$ and any increasing event $A$,
\begin{align}\label{eq:Russo-formula}
\frac{d}{dp}\mu_p(A)  &\geq \beta_p\sum_{x\in\cG}\frac{ \mu_p(A_x)}{	\sup_{\zeta\in\Omega}\mu^p_{x,\zeta}(1)} \geq \beta_p\sum_{x\in\cG}\mu_p(A_x)
\end{align}
\end{theorem}
\begin{remark}
In the case of Bernoulli product measures $\nu_p$, \eqref{eq:Russo-formula} simplifies to
$\frac{d}{dp}\nu_p(A) \geq \frac{1}{p}\sum_{x\in\cG}\nu_p(A_x)$,
which is in fact an equality and is the original form of Russo's formula.
\end{remark}

\begin{corollary}\label{cor:influences1}
Fix an increasing event $A\subset\Omega$. Let $\rho_p$ denote the log-Sobolev constant of $(L_p,\mu_p)$, which is assumed to be positive. Write $\delta_p:=\sup\{\mu^p_{x,\eta}(A_x) : \eta\in\Omega,x\in\cG\}$ and $\alpha_p:=\inf\left\{\mu^p_{x,\eta}(e) : \eta\in\Omega,x\in\cG, e\in\{0,1\}\right\}$. If $\delta_p<e^2\alpha_p^2$, then
\begin{align}
\frac{d}{dp}\mu_p(A) \geq \frac{\beta_p\log\left(e^2\alpha_p^2\delta_p^{-1}\right)}{4c_1C}\mu_p(A)(1-\mu_p(A)),
\end{align}
where $c_1$ and $C=C(\alpha_p^{-3},\abs{\cN},\rho_p)$ are as in Corollary \ref{cor:talagrand} and $\beta_p$ defined in Theorem \ref{thm:Russo}. 
\end{corollary}

\begin{corollary}\label{cor:influences2}
With the notation of Corollary \ref{cor:influences1} in place, write $\delta:=\sup_{p\in[a,b]}\delta_p$, $\alpha:=\inf_{p\in[a,b]}\alpha_p$, $\beta:=\inf_{p\in[a,b]}\beta_p$ and $\rho:=\inf_{p\in[a,b]}\rho_p$. For $a\leq p_1<p_2\leq b$, we have
\begin{align}
\mu_{p_1}(A)(1-\mu_{p_2}(A))\leq \left(\frac{\delta}{e^2\alpha^2}\right)^{(p_2-p_1)/C'},
\end{align}
where $C'=4c_1C(\alpha^{-3},\abs{\cN},\rho)/\beta$. 
\end{corollary}

Finally we can also obtain an analogue to the KKL bound \cite{KAHN:KALAI:LINIAL:88}. Here we slightly deviate from the setting of this section, in that we only have a single measure $\mu$, and $A$ can be any event in $\Omega$. For such an event we define $\supp(A)$ as the set of $x\in\cG$ for which $D_x \ind_A$ is not constant almost surely.
\begin{corollary}\label{cor:KKL}
Assume that $\alpha=\inf\{\mu_{x,\eta}(e) : \eta\in\Omega,x\in\cG,e\in\{0,1\}\}>0$ and $(L,\mu)$ satisfies a log-Sobolev inequality with constant $\rho>0$. 
Write $R:= \abs{\supp(A)}/[\mu(A)(1-\mu(A))]$, which we assume to be finite. Then
\[ \sup_{x\in\cG} \mu(A_x)\geq \min\left(\frac{\log(\frac{\alpha^4}{16}R)}{8CR},\frac{e^2\alpha^2}{2}\right), \]
where $C=C(\alpha^{-3},\abs{\cN},\rho)$ is as in Theorem \ref{thm:talagrand}.
\end{corollary}

\subsection{Organization of the proofs}
The proofs are organized as follows: Section \ref{section:talagrand} proves Talagrand's inequality by using Theorem \ref{thm:nice}. In Section \ref{section:Russo} we prove Russo's formula and related consequences plus Lemma \ref{lemma:IPS-K-condition}. Then in Section \ref{section:reverse-talagrand} we prove Theorem \ref{thm:reverse-talagrand}. The remaining Sections \ref{section:graphical-construction} to \ref{section:full-iteration} contain the proof of Theorem \ref{thm:nice}.


\section{Proof of Talagrand's Inequality}\label{section:talagrand}
We will follow essentially the proof of Talagrand's inequality as in \cite{CORDERO-ERAUSQUIN:LEDOUX:12} with the use of the commutation property given by Theorem \ref{thm:nice}. Before proving Theorem \ref{thm:talagrand} itself we need an auxiliary lemma.
\begin{lemma}\label{lemma:orlicz}
Let $\Phi(x)=\frac{x^2}{\log(e+\abs{x})}$. There is a numerical constant $C>0$ so that 
\begin{align}
 \int_1^2\norm{f}_r^2 \,dr \leq C \norm{f}_\Phi^2 .
\end{align}
\end{lemma}
\begin{proof}
By homogeneity we can assume $\norm{f}_\Phi=1$, which implies $\int \Phi(f)\,d\mu\leq 1$. 
Denote by $f_n=\abs{f}\ind_{e^{n-1}<\abs{f}\leq e^n}$, $n\geq 1$, and $g_0=\abs{f}\ind_{\abs{f}\leq 1}$. We have
\begin{align}\label{eq:orlicz-1}
 \sum_{n=0}^\infty \frac{1}{n+1}\int f_n^2 \,d\mu \leq  \sum_{n=0}^\infty \int \frac{f_n^2}{\log(e+f_n)}\,d\mu \leq 1.
\end{align}
To prove the claim, 
\begin{align}
 \int_1^2 \norm{f}_r^2\,dr 
 &= \int_1^2\left(\sum_{n=0}^\infty \int f_n^r\,d\mu\right)^{\frac2r}dr \\
 &\leq \int_1^2\left(\sum_{n=0}^\infty e^{(n-1)(r-2)}(n+1)\frac{\int f_n^2\,d\mu}{n+1}\right)^{\frac2r}dr.
\end{align}
By convexity and \eqref{eq:orlicz-1}, this is less than
\begin{align}
 \int_1^2 \sum_{n=0}^\infty e^{\frac2r(n-1)(r-2)}(n+1)^{\frac2r}\frac{\int f_n^2\,d\mu}{n+1}dr.
\end{align}
Once we have shown that there is a constant $C_1$ so that 
\begin{align}\label{eq:orlicz-2}
 \int_1^2 e^{\frac2r(n-1)(r-2)}(n+1)^{\frac2r}\,dr\leq C_1
\end{align}
 uniformly in $n\in\bN$ we can conclude that
$\int_1^2 \norm{f}_r^2\,dr \leq C_1.$
To show \eqref{eq:orlicz-2} we substitute $u=2/r$  and then use the fact that $\int_1^2 e^{2x(1-u)}x^u\,du \leq 1$ for all $x\geq1$:
\begin{align}
\int_1^2 e^{\frac2r(n-1)(r-2)}(n+1)^{\frac2r}\,dr &= \int_1^2 e^{2(n+1)(1-u)}(n+1)^{u}e^{-4(1-u)}\frac2{u^2}\,du \leq 2e^4. \qedhere
\end{align}
\end{proof}
\begin{proof}[Proof of Theorem \ref{thm:talagrand}]
Without loss of generality we assume that $\int f\,d\mu=0$. From the Poincar\'e inequality \eqref{eq:L2-decay} it follows that
\begin{align}
 \Var_\mu(f)=\norm{f}_2^2 &\leq \norm{f}_2^2-\norm{P_Tf}_2^2 + e^{-2\kappa T}\norm{f}_2^2\\ 
 &\leq \ldots\leq \frac{1}{1-e^{-2\kappa T}}(\norm{f}_2^2-\norm{P_Tf}_2^2) \\
 &= \frac{1}{1-e^{-2\kappa T}}\int_0^T 2\cE(P_tf, P_tf)\,dt \label{eq:talagrand-0}
\end{align}
Since $f$ is a good function, we can estimate \eqref{eq:talagrand-0} using Theorem \ref{thm:nice}, to get
\begin{align}
\Var_\mu(f) &\leq \frac{4K e^{72 \abs{\cN}^2(1+\abs{\cN})^2}}{1-e^{-2\kappa T}}\int_0^T 2^{ t} \sum_{x\in\cG}\norm{D_x f}^2_{p(t)} \,dt \\
&\leq \frac{4K e^{72 \abs{\cN}^2(1+\abs{\cN})^2} 2^{ T}}{1-e^{-2\rho T}}\sum_{x\in\cG} \int_0^T \norm{D_x f}^2_{p(t)} \,dt .	\label{eq:talagrand-1}
\end{align}
Choosing $T=\frac{1}{2\rho}$ and using the substitution $r=p(t)$, we have 
\[ \int_0^T \norm{D_x f}^2_{p(t)} \,dt \leq \frac{e}{2\rho}\int_1^2 \norm{D_x f}^2_r\,dr .\]
Using Lemma \ref{lemma:orlicz} we can conclude that
\begin{align}
 \eqref{eq:talagrand-1} &\leq C(\abs{\cN},\rho,K)\sum_{x\in\cG}\norm{D_x f}^2_\Phi. \qedhere
\end{align}
\end{proof}
\begin{proof}[Proof of Theorem \ref{thm:tensor-talagrand}]
 It is known that the log-Sobolev inequality tensorizes (see for example \cite[Theorem 4.4]{GUIONNET:ZEGARLINSKI:03}), that is, it generalizes to product spaces. The log-Sobolev constant on the product space is given by the minimum of the individual log-Sobolev constants. However, $\Omega$ is not of the form $E^\cG$, so we cannot apply Theorem \ref{thm:talagrand} directly. But the proofs are identical except for minor modifications to adjust to this more general setting. It suffices to replace the definition of $I$ in Section \ref{section:iteration-steps} by 
 \begin{align}
  I(T,v,S) &:= \sum_{i=1}^N\sup_{X\in \cX_T(\cG_i)}\sum_{x\in \cG_i} \normb{D_x \PsiSX[x+X](v,\infty) f}^2_{p\left(\uS(v)\right)},
 \end{align}
with $\cX_T(\cG)$ being the maps from $\oT$ to $\cG$, and also to replace $\abs{N}$ and $\rho$ by their respective maximum or minimum.
\end{proof}

\section{Russo's formula and other proofs}\label{section:Russo}
\begin{proof}[Proof of Lemma \ref{lemma:IPS-K-condition}]
Fix $x\in \cG$ and $f\in L^2(\mu)$. Let $\eta^*$ denote the maximizer (or one of the maximizers) of $f$ out of the set $\{\xi\in\Omega : \xi=\eta\text{ off }x\}$. Similarly let $\eta_*$ denote the minimizer. We have 
\begin{align}
\int (\Psi_xf-f)^2 \,d\mu 
&\geq \int \left(\int (f(\xi)-f(\eta))\mu_{x,\eta}(d\xi)\right)^2\ind_{\eta=\eta_*}\,\mu(d\eta)	\\
&\geq \int \left(\int_{\xi=\eta^*} (f(\xi)-f(\eta_*))\mu_{x,\eta}(d\xi)\right)^2\ind_{\eta=\eta_*}\,\mu(d\eta)	\\
&\geq \alpha^2 \int (f(\eta^*)-f(\eta_*))^2\ind_{\eta=\eta_*}\,\mu(d\eta).	
\end{align}
By \eqref{eq:Psi-invariant} and the fact that $\eta^*$ and $\eta_*$ do not depend on $\eta(x)$, the above is equal to
\begin{align}
 \alpha^2 \int (f(\eta^*)-f(\eta_*))^2 \Psi_x(\ind_{\eta=\eta_*})\,\mu(d\eta)	
&= \alpha^2 \int (f(\eta^*)-f(\eta_*))^2 \mu_{x,\eta}(\eta_*(x))\,\mu(d\eta)	\\
&\geq \alpha^3 \int \Psi_x(f-f(\eta))^2(\eta)\,\mu(d\eta).\qedhere
\end{align}
\end{proof}

\begin{proof}[Proof of Theorem \ref{thm:Russo}]
Fix $\epsilon>0$ so that $p-\epsilon>a$. Since $\mu_p$ dominates $\mu_{p-\epsilon}$, we can construct a coupling $\hP$ of $\mu_p$ and $\mu_{p-\epsilon}$ so that $\eta$ and $\xi$ have law $\mu_p$ and $\mu_{p-\epsilon}$ respectively, and $\eta\geq \xi$ $\hP$-a.s. 

Fix a finite subset $\Lambda$ of $\cG$, and set $\Delta:=\{x\in\cG: \eta(x)\neq \xi(x)\}$. Since $A$ is increasing,
\begin{align}
\mu_p(A)-\mu_{p-\epsilon}(A)
= \hP\left(\eta\in A, \xi \not\in A\right)	
&= \sum_{x\in\Lambda} \hE\left( \frac{1}{\abs{\Delta\cap\Lambda}}\ind_{x\in\Delta}\ind_{\eta\in A, \xi\not\in A} \right)	\\
&\geq \sum_{x\in\Lambda} \hE\left( \frac{1}{\abs{\Delta\cap\Lambda}}\ind_{x\in\Delta}\ind_{\eta\in A,\eta^x\not\in A} \right).
\end{align}
For $x\in \Lambda$ we have $\ind_{x\in\Delta}\frac{1}{\abs{\Delta\cap\Lambda}}\geq \ind_{\Delta\cap\Lambda=\{x\}}$, and since $A$ is increasing, $A_x=\{\eta\in A, \eta^x\not\in A\}$ implies $\eta(x)=1$. Hence,
\begin{align}
\hE\left( \frac{1}{\abs{\Delta\cap\Lambda}}\ind_{x\in\Delta}\ind_{A_x} \right) 
&\geq \hE\left(\ind_{A_x}  \hP\left[\Delta\cap\Lambda=\{x\}\;\middle|\;\eta,\eta(x)=1\right]\right)	\\
&\geq \hP\left(A_x\right) \inf_{\zeta\in\Omega}\hP\left(\Delta\cap\Lambda=\{x\}\;\middle|\;\eta=\zeta\text{ off }x,\eta(x)=1\right).
\end{align}
Fix $\zeta\in\Omega$ for now. By using \eqref{eq:mu-is-glauber} and assuming that $x$ is such that $x+\cN\subset \Lambda$, we get 
\begin{align}
&\hP\left(\Delta\cap\Lambda=\{x\}\;\middle|\;\eta=\zeta\text{ off }x,\eta(x)=1\right) \\
&\quad= \frac{1}{\mu^p_{x,\zeta}(1)}\hP\left(\Delta\cap\Lambda=\{x\}, \eta(x)=1\;\middle|\;\eta=\zeta\text{ off }x\right)	\\
&\quad= \frac{1}{\mu^p_{x,\zeta}(1)}\hE\left(\ind_{\Delta\cap\Lambda\setminus\{x\}=\emptyset}\hP\left( \eta(x)=1,\xi(x)=0 \;\middle|\;\eta=\xi=\zeta\text{ off }x\right)\;\middle|\;\eta=\zeta\text{ off }x\right)	\\
&\quad \geq \frac{1}{\mu^p_{x,\zeta}(1)}\hE\left(\ind_{\Delta\cap\Lambda\setminus\{x\}=\emptyset}\;\middle|\;\eta=\zeta\text{ off }x\right)\inf_{\zeta'\in\Omega}\hP\left( \eta(x)=1,\xi(x)=0 \;\middle|\;\eta=\xi=\zeta'\text{ off }x\right).
\end{align}
By using \eqref{eq:mu-is-glauber} as well as $\eta(x)\geq \xi(x)$, we have
\begin{align}
\inf_{\zeta'\in\Omega}\hP\left( \eta(x)=1,\xi(x)=0 \;\middle|\;\eta=\xi=\zeta'\text{ off }x\right) = \inf_{\zeta'\in\Omega} \left[\mu^p_{x,\zeta'}(1)-\mu^{p-\epsilon}_{x,\zeta'}(1)\right] \geq \beta_p\epsilon-o(\epsilon). \label{eq:russo-1}
\end{align}
By using the fact that $\Lambda$ is finite and $\mu^{p-\epsilon}$ converges to $\mu^p$ we also have
\begin{align}\label{eq:russo-2}
\hE\left(\ind_{\Delta\cap\Lambda\setminus\{x\}=\emptyset}\;\middle|\;\eta=\zeta\text{ off }x\right) = 1-o(1)
\end{align}
uniformly in $\zeta$. Combining \eqref{eq:russo-1} and \eqref{eq:russo-2}, we get
\begin{align}
\inf_{\zeta\in\Omega}\hP\left(\Delta\cap\Lambda=\{x\}\;\middle|\;\eta=\zeta\text{ off }x,\eta(x)=1\right) &\geq \frac{1}{\sup_{\zeta\in\Omega}\mu^p_{x,\zeta}(1)}\beta (\epsilon -o(\epsilon))\ind_{x+\cN\subset\Lambda}.
\end{align}
Therewith is
\begin{align}
\mu_p(A)-\mu_{p-\epsilon}(A) &\geq \sum_{x\in\Lambda:x+\cN\in\Lambda}\mu_p(A_x)\frac{\beta}{\sup_{\zeta\in\Omega}\mu^p_{x,\zeta}(1)} (\epsilon -o(\epsilon)),
\end{align}
and by dividing by $\epsilon$ and then sending first $\epsilon$ to 0 and then $\Lambda$ to $\cG$ we finish the proof.
\end{proof}
To combine Russo's formula and Talagrand's inequality we require the following lemma.
\begin{lemma}\label{lemma:Dx-to-influences}
For any event $A$, any $\mu$ out of the measures considered in Section \ref{section:influences}, any $q\geq 1$ and $x\in\cG$,
\begin{align}
\left(\inf_{\eta\in\Omega}\mu_{x,\eta}(0)\right)^q\mu(A_x)\leq \norm{D_x\ind_A}_q^q \leq 2\mu(A_x).
\end{align}
\end{lemma}
\begin{proof}
For any $q\geq 1$, we use Jensen's inequality and drop the $q$:
\begin{align}
\norm{D_x\ind_A}_q^q &\leq \int\int \abs{\ind_{\xi\in A}- \ind_{\eta\in A}}\,\mu_{x,\eta}(d\xi)\mu(d\eta) 
\leq \int\int (\ind_{\xi\in A_x} + \ind_{\eta\in A_x})\,\mu_{x,\eta}(d\xi)\mu(d\eta). 
\end{align}
Since $\xi$ is $\mu$-distributed the upper bound follows.

The lower bound is obtained by the following calculation:
\begin{align}
\norm{D_x\ind_A}_q^q 
&\geq \int\abs{\int \left(\ind_{\xi\in A}- \ind_{\eta\in A}\right)\,\mu_{x,\eta}(d\xi)}^q \ind_{\eta\in A_x} \mu(d\eta)	\\ 
&= \int \ind_{\eta\in A_x}\mu_{x,\eta}(0)^q \mu(d\eta) \geq \inf_{\eta\in\Omega}\mu_{x,\eta}(0)^q\mu(A_x).\qedhere
\end{align}
\end{proof}

\begin{proof}[Proof of Corollary \ref{cor:influences1}]
Lemma \ref{lemma:IPS-K-condition} guarantees that all conditions of Theorem \ref{thm:talagrand} are satisfied. Combining Theorem \ref{thm:Russo}, Lemma \ref{lemma:Dx-to-influences} and Corollary \ref{cor:talagrand} yields the claim.
\end{proof}

\begin{proof}[Proof of Corollary \ref{cor:influences2}]
Using Corollary \ref{cor:influences1}, the proof is the same as Corollary 1.3 in \cite{TALAGRAND:94}.
\end{proof}

\begin{proof}[Proof of Corollary \ref{cor:KKL}]
We can assume that $\sup_{x\in\cG}\mu(A_x) \leq \alpha^2e^2/2$, otherwise there is nothing to prove. 

For the other estimate, we follow the argument from \cite{CORDERO-ERAUSQUIN:LEDOUX:12}. Let $x^*$ be the (or one of the) maximizer(s) of $\sup_{x\in\cG}\mu(A_x)$. If $\mu(A_{x^*})>R^{-\frac12}$, then the result holds as $C \geq 1$. So we assume $\mu(A_{x^*})\leq R^{-\frac12}$. By using Corollary \ref{cor:talagrand} and Lemma \ref{lemma:Dx-to-influences},
\begin{align}
\mu(A)(1-\mu(A))
&\leq C\abs{\supp(A)}\frac{\norm{D_{x^*}\ind_A}_2^2}{1+\log\left(\norm{D_{x^*}\ind_A}_2/\norm{D_{x^*}\ind_A}_1\right)}	\\
&\leq C\abs{\supp(A)} \frac{2\mu(A_{x^*})}{1+\log\left(\frac{\alpha}{2}\mu(A_{x^*})^{-\frac12}\right)}.
\end{align}
By using the upper bounds on $\mu(A_{x^*})$, the result follows.
\end{proof}

\section{Log-Sobolev inequality from Talagrand's inequality}\label{section:reverse-talagrand}
To prove Theorem \ref{thm:reverse-talagrand} we follow the proof of Proposition 1 in \cite{BOBKOV:HOUDRE:99}, where the analogue statement was shown for continuous spaces. Our discrete derivative operators $D_x$ force slightly more complicated arguments, but the general structure of the proof is unchanged. In contrast to the previous sections here we rely on various Orlicz norms and their properties. For basic properties of Orlicz spaces the reader is referred for example to \cite{KRASNOSELSKII:RUTITSKII:61, LEONARD:07}.

We say $\varphi:\bR\to\bR\cup\{\infty\}$ is a Young function if it is convex, even, $\varphi(0)=0$, and $\lim_{x\to\infty}\varphi(x)=\infty$.
The norm $\norm{\cdot}_\varphi$ is defined just as in \eqref{eq:Orlicz-norm}, that is
\[ \norm{f}_\varphi := \inf\left\{ a>0 : \int \varphi\left(\frac{f}{a}\right)\,d\mu\leq 1 \right\}. \] 
\begin{lemma}\label{lemma:Orlicz-norm-estimate}
 Let $\varphi$ be a Young function and assume $\int \varphi(f)\,d\mu\leq K$. Then $\norm{f}_\varphi\leq \max(1,K)$.
\end{lemma}
\begin{proof}
 If $K\leq 1$, then $\int \varphi( f)\,d\mu \leq 1$.
 If $K>1$, by convexity, 
 \[ \int \varphi\left(\frac f K \right)\,d\mu \leq \frac1K \int  \varphi(f)\,d\mu \leq 1. \qedhere \]
\end{proof}
\begin{lemma}\label{lemma:Orlicz-Phi-2}
 Let $f\in L^2(\mu)$. Then
 \[ \norm{f}_\Phi \leq 2\norm{f}_2 .\]
\end{lemma}
\begin{proof}
 By homogeneity assume $\norm{f}_2=1$. As $\Phi(x)\leq x^2+\abs{x}$, we have $\int \Phi(f)\,d\mu\leq 2$, and Lemma \ref{lemma:Orlicz-norm-estimate} completes the proof.
\end{proof}
\begin{lemma}\label{lemma:Orlicz-Phi-Holder}
 For measurable $f,g:\widetilde\Omega\to\bR$ on some Polish space $\widetilde\Omega$, $\norm{fg}_\Phi\leq 24\norm{f}_{e^{x^2}-1}\norm{g}_2$.
\end{lemma}
For a proof, see \cite{BOBKOV:HOUDRE:99}, Lemma 7.
\begin{lemma}\label{lemma:Orlicz-Dx}
 Let $(\mu\times\mu_x)$ denote the probability measure on $\Omega\times\Omega$ given by $(\mu\times\mu_x)(d\eta,d\xi)=\mu_{\eta,x}(d\xi)\mu(d\eta)$, and for $f:\Omega\to\bR$ let $f_2:\Omega\times\Omega\to\bR$, $f_2(\eta,\xi) = f(\xi)-f(\eta)$. Then
 \[ \norm{D_x f}_{\Phi;\mu} \leq \norm{f_2}_{\Phi;(\mu\times\mu_x)}, \]
 where the Orlicz norms are to be understood on the measure spaces $(\Omega, \mu)$ and $\left(\Omega\times\Omega, (\mu\times\mu_x)\right)$ respectively.
\end{lemma}
\begin{proof}
 By homogeneity assume $\norm{f_2}_{\Phi;(\mu\times\mu_x)}=1$, hence $\int \Phi(f_2)\,d(\mu\times\mu_x)\leq 1$.
 For $D_x f$ we have, using Jensen's inequality,
 \begin{align}
  \int \Phi(D_x f)\,d\mu 
  &= \int \Phi\left(\int (f(\xi)-f(\eta)) \,d\mu_{x,\eta}\right) \,\mu(d\eta)	\\
  &\leq \int \int \Phi(f(\xi)-f(\eta)) \,\mu_{x,\eta}(d\xi)\mu(d\eta)	\\
  &= \int \Phi(f_2) \,d(\mu\times\mu_x) =1.
 \end{align}
 Therefore $\norm{D_x f}_{\Phi;\mu} \leq 1$.
\end{proof}

\begin{proof}[Proof of Theorem \ref{thm:reverse-talagrand}]
Let $\varphi(x):=x^2\log(1+x^2)$. In \cite{BOBKOV:GOETZE:99}, Proposition 4.1, it is shown that $\sup_{c\in\bR}\Ent_\mu\left((f+c)^2\right) \leq \frac{13}{4} \norm{f-\int f\,d\mu}_\varphi^2$.
 To prove \eqref{eq:log-Sobolev} it therefore suffices to show that $\norm{f-\int f\,d\mu}_\varphi^2\leq C'\cE(f,f)$.
 
 Let $f$ be bounded, $\int f\,d\mu=0$ and $\cE(f,f) = \frac{1}{2C}$, where $C$ the constant from \eqref{eq:reverse-talagrand-assumption}. By \eqref{eq:reverse-talagrand-assumption}, Lemma \ref{lemma:Orlicz-Phi-2} and Jensen's inequality,
 \begin{align}\label{eq:reverse-talagrand-1}
  \int f^2\,d\mu &\leq C\sum_{x\in\cG}\norm{D_x f}^2_\Phi \leq 4 C \sum_{x\in\cG}\norm{D_x f}_2^2 \leq 8C \cE(f,f) = 4.
 \end{align}
 Define $g:\Omega\to\bR$ as $g(\eta)=f(\eta)\sqrt{\log(1+f(\eta)^2)}$. By Cauchy-Schwarz's inequality and \eqref{eq:reverse-talagrand-1},
 \begin{align}
  \left(\int g \,d\mu\right)^2 &\leq \int f^2 \,d\mu \int \log(1+f^2)\,d\mu \\
  &\leq \int f^2 \,d\mu \log\left(1+\int f^2\,d\mu\right) \leq 4\log5. \label{eq:reverse-talagrand-2}
 \end{align}
 Since $\abs{\left(x\sqrt{\log(1+x^2)}\right)'}\leq 2\sqrt{\log(1+x^2)}$, we can use the mean value theorem to obtain the estimate
 \begin{align}\label{eq:reverse-talagrand-g}
  \abs{g(\eta)-g(\xi)}\leq 2\sqrt{\log(1+\max(f(\eta)^2,f(\xi)^2))}\abs{f(\eta)-f(\xi)}.
 \end{align}
Using Lemma \ref{lemma:Orlicz-Dx}, \eqref{eq:reverse-talagrand-g} and Lemma \ref{lemma:Orlicz-Phi-Holder},
\begin{align}
 \norm{D_x g}_\Phi 
 &\leq \norm{g_2}_{\Phi;(\mu\times\mu_x)}	\\
 &\leq \norm{2\sqrt{\log(1+\max(f(\eta)^2,f(\xi)^2))}\abs{f(\xi)-f(\eta)}}_{\Phi;(\mu\times\mu_x)}	\\
 &\leq 48 \norm{\sqrt{\log(1+\max(f(\eta)^2,f(\xi)^2))}}_{e^{x^2}-1;(\mu\times\mu_x)}	\\
 &\qquad\cdot\norm{f(\xi)-f(\eta)}_{L^2(\mu\times\mu_x)}
\end{align}
We have $\norm{f(\xi)-f(\eta)}_{L^2(\mu\times\mu_x)}^2 = \int \Psi_x(f-f(\eta))^2(\eta)\,\mu(d\eta)$.
To estimate the Orlicz norm, using $\max(a^2,b^2)\leq 2a^2+2(b-a)^2$, \eqref{eq:reverse-talagrand-1} and $\cE(f,f)\leq \frac{1}{2C}<\frac12$,
\begin{align}
 &\int (e^{x^2}-1)\circ \sqrt{\log(1+\max(f(\eta)^2,f(\xi)^2))} (\mu\times\mu_x)(d\eta d\xi)	\\
 &= \int \int \max(f(\eta)^2,f(\xi)^2) \,\mu_{\eta,x}(d\xi)\mu(d\eta)	\\
 &\leq 2\int f(\eta)^2\mu(d\eta)+ 2\int \int (f(\xi)-f(\eta))^2 \,\mu_{\eta,x}(d\xi)\mu(d\eta)	\\
 &\leq 8 +2 \cE(f,f) \leq 10.
\end{align}
Hence, by Lemma \ref{lemma:Orlicz-norm-estimate}, $\norm{\sqrt{\log(1+\max(f(\eta)^2,f(\xi)^2))}}_{e^{x^2}-1;(\mu\times\mu_x)}\leq 10$.
In total we obtain the estimate
\begin{align}\label{eq:reverse-talagrand-3}
 \sum_{x\in\cG} \norm{D_x g}_\Phi^2 \leq 480^2 \sum_{x\in\cG} \int \Psi_x(f-f(\eta))^2(\eta)\,\mu(d\eta) = 2\cdot 480^2 \cE(f,f).
\end{align}
To finish the proof, using \eqref{eq:reverse-talagrand-2}, \eqref{eq:reverse-talagrand-assumption} and \eqref{eq:reverse-talagrand-3}, 
\begin{align}
 \int \varphi(f)\,d\mu 
 &= \int g^2\,d\mu = \Var_\mu(g)+\left(\int g \,d\mu\right)^2	\\
 &\leq 4\log5 + C\sum_{x\in\cG}\norm{D_x g}^2_\Phi	\\
 &\leq 4\log5+2\cdot 480^2 < 700^2.
\end{align}
By Lemma \ref{lemma:Orlicz-norm-estimate}, $\norm{f}_\varphi\leq 700^2$. For general $f$ we therefore have
\begin{align}
 \Ent_\mu(f^2) \leq \frac{13}{4} \norm{f-\int f\,d\mu}^2_\varphi < \frac{13}{4} 700^4 C\cE(f,f).
\end{align}
\end{proof}

\section{Graphical construction}\label{section:graphical-construction}
In preparation of the proof of Theorem \ref{thm:nice} we start with the basic graphical construction of the dynamics. It is classical that a wide range if interacting particle systems and related models admit a graphical construction. Based on one a Poisson point process on $\cG$ and additional randomness one can construct the law of the process. This connection is widely used. We will also make use of this underlying structure, which is one of the reasons we required that the generator $L$ is of the special form \eqref{eq:generator}. Intuitively speaking, the law of the Markov process is given by a Poisson point process of intensity 1, where at each point a resampling event takes place, that is a configuration $\eta$ is replaced by a configuration $\xi$, where $\xi$ is drawn independently from $\mu_{x,\eta}$. We will need some more subtle facts of the graphical construction, and in this section we will introduce the necessary notation and rigorous definition.

As a motivation, note that from \eqref{eq:Psix} follows that the generator $L$ can be written as
\[ Lf  = \sum_{x\in\cG} [\Psi_x f -f]. \]
Less immediate is the fact that the semi-group $P_t$ can also be written in terms of $\Psi_x$. This is our aim, but to do so requires a few steps.

 For a finite set $A=\{(x_i,t_i) : 1\leq i \leq n, t_1 < t_2 < ... < t_n \}\subset\cG\times[0,T]$ we can define $\Psi_A$ as
\[ \Psi_A:=\Psi_{x_1}\Psi_{x_2}\cdots\Psi_{x_n}. \]
To extend this definition to infinite sets $A$ requires some more work, as we no longer have a well-ordered sequence. Before going into the details we state the aim of this section, which is how to represent the semi-group via a Poisson point process and $\Psi$:
\begin{lemma}\label{lemma:graphical-construction-b}
Let $\oNt$ be a Poisson point process with unit intensity on $\cG\times[0,t]$. The semi-group $(P_t)$ is given by 
\begin{align}
 P_t f = \int \Psi_\oNt f \,d\oNt. 
\end{align} 
\end{lemma}

Let $A$ be a countable subset of $\cG\times[0,T]$ with no two points at the same time. A partial order $<_A$ on $\cG\times[0,T]$ is defined as follows: $(x,t)<_A (y,s)$ iff either $x=y$ and $t<s$ or there exists a finite subset $\{(x_1,t_1),\ldots(x_K,t_K)\}\subset A$ such that $t<t_{1}<t_{2}<\ldots<t_{K}\leq s$ and $x_{{m-1}} \in x_{m} + \cN$, $2\leq m \leq K$, as well as $x\in x_{1}+\cN$ and $x_{K}=y$. We write $A_{<x}:=\{(y,t)\in A:(y,t)<_A (x,T)\}$. We call $A$ \emph{locally finite}, if 
$\abs{A_{<x}}<\infty$ for all $x\in\cG$. 

For a locally finite $A\subset \cG\times[0,T]$ and a local function $f:\Omega\to\bR$ we define
\begin{align}
 \Psi_A f &:= \Psi_{A_f} f,\qquad A_f:=\bigcup_{x\in\supp(f)}A_{<x}.
\end{align}
To extend this definition to non-local functions, note that $\Psi_A$ induces a probability measure $\nu_{A,\eta}$ on $\Omega$ by $\int f \,d\nu_{A,\eta} = \Psi_A f(\eta)$, $f$ local. We then define $\Psi_A$ for non-local functions via 
\begin{align}
 \Psi_A f(\eta) := \int f \,d\nu_{A,\eta}.
\end{align}
\begin{lemma}\label{lemma:graphical-construction-a}
Let $A,B$ be locally finite sets so that $A\leq_{A\cup B} B$, that is for no $a\in A, b\in B$ the relation  $b<_{A\cup B} a$ holds. Then
\begin{align}
 \Psi_{A\cup B} = \Psi_A \Psi_B. \label{eq:graphical-1}
\end{align}
\end{lemma}
\begin{proof}
Let $f$ be a local function. By the nearest-neighbour property of $\Psi_x$ and the fact that $B$ is locally finite the function $\Psi_B f$ is local as well, with support given by
\begin{align}
 \supp(\Psi_B f) = \left\{x \in \cG : (x,0)<_B (y,T), y\in\supp(f)\right\}.
\end{align}
By construction 
\begin{align}
 \Psi_A\Psi_Bf = \Psi_{A_{\Psi_B f}}\Psi_{B_f} f .
\end{align}
From $A\leq_{A\cup B} B$ it then follows that
\begin{align}
 A_{\Psi_B f} = \left\{ (x,s)\in A : (x,s) <_{A\cup B} (y,T), y\in\supp(f) \right\}.
\end{align}
Hence 
\begin{align}
 \Psi_A\Psi_B f = \Psi_{A_{\Psi_B f}\cup B_f} f= \Psi_{(A\cup B)_f} f = \Psi_{A\cup B} f. 
\end{align} 
\end{proof}
We now have the tools to express the semi-group $P_t$ via the $\Psi_x$.
\begin{proof}[Proof of Lemma \ref{lemma:graphical-construction-b}]
Define
\[ \widetilde{P}_t := \int \Psi_\oNt f \,d\oNt .\]
Lemma \ref{lemma:graphical-construction-a} applied to $\oNt$ and $\oN_{[t,t+s]}$ shows that the operator family
is a semi-group, and easy calculations show that it is in fact a Markov semi-group whose generator is given by
\[ \widetilde{L}f = \sum_{x\in\cG} (\Psi_x f - f). \]
Since $\widetilde{L}=L$ the semi-groups must be the same as well.
\end{proof}

\section{Binary trees}\label{section:trees}
In this short section we will introduce binary trees and associated notation. The way binary trees are tied to the proof of Theorem \ref{thm:nice} is delegated to Sections \ref{section:iteration-steps} and \ref{section:full-iteration}.

A (rooted) \emph{tree} is a finite connected cycle-free graph with a distinguished root vertex, which we denote by 0. Let $d$ be the graph distance. Given a vertex $v$ in a tree $T$, an adjacent vertex $w\in T$ is said to be a \emph{child} of $v$ if $d(w,0)=d(v,0)+1$. Conversely, the \emph{parent} of a non-root vertex $v$ is the unique adjacent vertex $w$ with $d(w,0)=d(v,0)-1$. We call the the \emph{descendants} of a vertex $v$ the set of all vertices $w$ whose unique path to the root passes through $v$.

A \emph{binary tree} is a tree whose vertices have at most 2 children, with the additional information whether a child is the \emph{left} child or the \emph{right} child (i.e., a vertex can have no children, a left child, a right child, or a left and a right child). We call a binary tree \emph{full} if for each vertex the number of children is either 0 or 2. A vertex with no children is called a \emph{leaf}. With $\partial T$ we denote the set of all leaves of a tree $T$, and with $\oT=T\bs \partial T$ the set of interior vertices.

From now on we will only consider full binary trees. On such a tree $T$ we can introduce a well-ordering: 
\begin{wrapfigure}{r}{0.3\textwidth}
{\Tree [.0 [.-2 -3 -1 ] [.4 [.2 1 3 ] 5 ] ]}

\emph{A small tree with vertices identified by their embedding into $\bZ$}
\end{wrapfigure}
For any $v\in T$, the left child and all of its descendants are smaller than $v$, and the right child and all of its descendants are bigger than $v$. 

From this well-ordering we also obtain an embedding $\iota_\bZ$ of $T$ into $\bZ$, where we assume that $\iota_\bZ$ is order preserving, $\iota_\bZ(0)=0$, and $\iota_\bZ(T)=\{a,a+1,\dots,b\}$ for some $a,b\in\bZ$. Observe that in this ordering leaf vertices and interior vertices alternate, with leaves getting mapped onto odd numbers (with the exception of the tree where the root is already a leaf). 

Based on the embedding $\iota_\bZ$ we define what we mean by $v+1$: 
\begin{align*}
v+1 &:= \begin{cases}
         \iota_\bZ^{-1}(\iota_\bZ(v)+1), &\iota_\bZ(v)<b;	\\
         \infty, &\iota_\bZ(v)=b.
        \end{cases} 
\end{align*}
Note that we simply write $\infty$ for a virtual vertex to the right of all of $T$. Similarly we define $v-1$, with $-\infty$ for a virtual vertex to the left of $T$.

Let $\cT_n$ denote the set of full binary trees with exactly $n$ leaves, and $\cT=\bigcup_{n\geq 1} \cT_n$ the set of all full binary trees. Let $l(v)$ denote the left child of a vertex $v\in \oT$, and $r(v)$ its right child.

\section{\texorpdfstring{$T$}{T}-partitions and the basic iteration steps}\label{section:iteration-steps}
Now we will use binary trees to develop the core steps of the iteration scheme. Given a full binary tree $T$ and a real number $t\geq 0$ we define a $T$-partition of $[0,t]$ as a map $S:T \to [0,\infty[$ with $S(0)=t$ and $S(l(v))+S(r(v))=S(v)$ for all $v\in\oT$. Such a map represents a hierarchical partition: The original interval $[0,t]$ corresponds to the root vertex, and the interval is then split into two pieces where their respective lengths are represented by the values $S(l(0))$ and $S(r(0))$. Those intervals are then split further, with each $v\in \oT$ corresponding to a further refinement of the partition. The leaves encode the final partition, and for a vertex $v\in T$ we write $\uS(v):= \sum_{w\in\partial T, w<v} S(w)$ and $\oS(v):= \sum_{w\in\partial T, w\leq v} S(w)$. In other words, the partition of the interval $[0,t]$ is given by $0=\uS(v_*)\leq \uS(v_*+2) \leq \uS(v_*+4) \leq ... \leq \uS(v^*) \leq \uS(\infty)=t$, where $v_*=\min\{v: v\in \partial T\}$, $v^*=\max\{v: v\in \partial T\}$

The set of all possible $T$-partitions of $[0,t]$ we call $\cS_{T,t}$. In addition to the $T$-partitions we also need maps $X:\oT\to\cG$, which we collect in the set $\cX_T$. The set $\cS_{T,t}$ plays a big role in the structure of the proofs. The contribution of $\cX_T$ is minor and is mostly used to simplify notation.

From now on we fix $t\geq 0$. Remember the operators $\Psi_A$ from the graphical construction. 
Given a tree $T$ as well as $v\in T$, $S \in \cS_{T,t}$, $X\in\cX_T$ and a realization of a Poisson point process $\oNt\subset \cG\times[0,t]$ we define a linear operator $\PsiSX(v,\oNt)$ on $L^p(\mu)$ by
\begin{align*}
  \PsiSX\left(v,\oNt\right)f = \begin{cases}
                   \Psi_{X(v)}f,\quad& v\in \oT;\\
                   \Psi_{\oN_{[\uS(v),\oS(v)[}}f, &v\in\partial T.
                  \end{cases}
\end{align*}
Based on $\PsiSX(v,\oNt)$, we write 
\[ 
\PsiSX(v,w,\oNt) = 
\begin{cases}
\PsiSX(v,\oNt)\PsiSX(v+1,\oNt)\dots\PsiSX(w,\oNt),\quad&v\leq w;\\
\Id, &v>w.
\end{cases} 
\]
 For notational convenience we set $\PsiSX(v,\infty,\oNt) = \PsiSX(v, \max\{ w : w\in T\},\oNt)$, $\PsiSX(-\infty,w,\oNt) = \PsiSX(\min\{v:v\in T\},w,\oNt)$ and $\PsiSX(\infty,\infty,\oNt)=\Id$.

For a function $g(\oNt,\eta)$ of a realization $\oNt$ of a Poisson point process on $\cG\times [0,t]$ and $\eta\in\Omega$, we define for $p\geq1$ the norm
\begin{align}
\normb{g}_{p} &:= \norm{\int \abs{g(\oNt,\cdot)} d\oNt}_p.
\end{align}

Now we define the key object to be used in the proof of Theorem \ref{thm:nice}.
Given a tree $T$, as well as $v\in T$ and $S\in \cS_{T,t}$, define
\begin{align*}
 I(T,v,S) &:= \sup_{X\in \cX_T}\sum_{x\in \cG} \normb{D_x \PsiSX[x+X](v,\infty,\cdot) f}^2_{p(\uS(v))}.
\end{align*}
For a tree $T\in\cT_{n+1}$ the expression $I(T,v,S)$ represents an error term of order $n$. The vertex $v$ is used to enumerate the different error terms which all correspond to the same tree $T$. Note that $I(T,\infty,S)=\sum_{x\in\cG}\norm{D_x f}_{p(t)}^2$, which is the goal for our estimation. The starting point, $\sum_{x\in\cG}\norm{D_x P_t f}_2^2$, can also be expressed by $I$: 
\begin{align}
\sum_{x\in\cG}\norm{D_x P_t f}_2^2 \leq I(\{0\},0,S)=I(\{0\},-\infty,S). \label{eq:I-at-0}
\end{align}
 This estimate follows from 
\begin{align}
\abs{D_x P_t f} &= \abs{ \int D_x \Psi_{\oNt} f d\oNt } \leq \int \abs{ D_x \Psi_{\oNt} f  } d\oNt,
\end{align}
which is obtained by Lemma \ref{lemma:graphical-construction-b} and Jensen's inequality.

The role of $S$ is not yet evident at this point, but it will become more clear in Lemma \ref{lemma:leaf} below.

In the following we will prove two lemmas which are the basic building blocks for the argument. Both lemmas estimate $I(T,v,S)$ by $I(T,v+1,S)$, but the estimates are different for leaves and interior vertices. If $v$ is an interior vertex, the statement is straightforward.
\begin{lemma}\label{lemma:interior}
Let $T\in\cT_n,n\geq2$, and $v\in \oT$, $S\in\cS_{T,t}$. Then
\[ I(T,v,S)\leq 9 I(T,v+1,S). \]
\end{lemma}
The more complicated step is when $v$ is a leaf. In that case we have to introduce an extra term with a tree of a higher order, which is the reason for introducing the trees in the first place.
\begin{lemma}\label{lemma:leaf}
Let $T\in \cT_n$ and $v\in\partial T$, $S\in \cS_{T,t}$. Then
\[ I(T,v,S) \leq 2 I(T,v+1,S) + 2\abs{\cN}^2 S(v)(1+S(v)\abs{\cN})^2\int_0^{S(v)}I(T'_v, v'-1,S'_{r,v})\,dr, \]
where
\begin{itemize}
\item $T'_v$ is the tree in $\cT_{n+1}$ obtained by appending two leaves to the vertex $v$;
\item a vertex $w'\in T'_v$ represents the image of $w\in T$ when $T$ is embedded canonically into $T'_v$;
\item $S'_{r,v} \in \cS_{T'_v,t}$ is given by $S'_{r,v}(v'-1) = r$, $S'_{r,v}(v'+1) = S(v)-r$ and $S'_{r,v}(w') = S(w)$ for $w\in T$.
\end{itemize}
\end{lemma}
\begin{corollary}\label{cor:leaf}
In the setting of Lemma \ref{lemma:leaf}, if $t\leq 1$, then
\[ I(T,v,S) \leq 2 I(T,v+1,S) + 2\abs{\cN}^2(1+\abs{\cN})^2 S(v)\int_0^{S(v)}I(T'_v, v'-1,S'_{r,v})\,dr. \]
\end{corollary}
The remainder of this section is dedicated to the proofs of Lemmas \ref{lemma:interior} and \ref{lemma:leaf}.
\begin{proof}[Proof of Lemma \ref{lemma:interior}]
As $v$ is an interior vertex,
\begin{align*}
 D_x \PsiSX(v,\infty,\cdot)f &= D_x \Psi_{X(v)}\PsiSX(v+1,\infty,\cdot) f \\
 &= \Psi_x \Psi_{X(v)}\PsiSX(v+1,\infty,\cdot) f - \Psi_{X(v)}\PsiSX(v+1,\infty,\cdot) f .
\end{align*}
By adding and subtracting the appropriate terms and using the triangle inequality,
\begin{align}
 &\normb{\Psi_x \Psi_{X(v)}\PsiSX(v+1,\infty,\cdot) f - \Psi_{X(v)}\PsiSX(v+1,\infty,\cdot) f}_p \nonumber\\
   &\quad\leq \normb{\Psi_x \Psi_{X(v)}\PsiSX(v+1,\infty,\cdot) f - \Psi_{x}\PsiSX(v+1,\infty,\cdot) f}_p \label{eq:interior-a}\\
  &\qquad+\normb{\Psi_{x}\PsiSX(v+1,\infty,\cdot) f - \PsiSX(v+1,\infty,\cdot) f}_p \label{eq:interior-b}\\
  &\qquad+\normb{\Psi_{X(v)}\PsiSX(v+1,\infty,\cdot) f - \PsiSX(v+1,\infty,\cdot) f}_p \label{eq:interior-c}.
\end{align}
Note that the terms in \eqref{eq:interior-b} and \eqref{eq:interior-c} are $D_x\PsiSX(v+1,\infty,\cdot) f$ and $D_{X(v)}\PsiSX(v+1,\infty,\cdot) f$ respectively. Furthermore, with $z=0$ or $z=X(v)$,
\begin{align*}
&\sup_{X\in\cX_T}\sum_{x\in\cG} \normb{D_{x+z}\PsiSX[x+X](v+1,\infty,\cdot) f}_{p\left(\uS(v)\right)}^2\\
&= \sup_{X\in\cX_T}\sum_{x\in\cG} \normb{D_{x}\PsiSX[x-z+X](v+1,\infty,\cdot) f}_{p\left(\uS(v)\right)}^2 \\
&\leq I(T,v+1,S).
\end{align*}
The last inequality uses the fact that for interior vertices, $\uS(v)=\uS(v+1)$. Note that by \eqref{eq:Psi-invariant} and Jensen's inequality, $\normb{\Psi_x g}_p\leq \normb{g}_p$. Applying this fact to the function $D_{X(v)}\PsiSX f$ we can estimate \eqref{eq:interior-a} the same way as \eqref{eq:interior-b} and \eqref{eq:interior-c}.
To put everything together, we use above estimates and the fact that $(a+b+c)^2\leq 3(a^2+b^2+c^2)$:
\begin{align*}
 &\sup_{X\in\cX_T}\sum_{x\in\cG} \normb{D_x \PsiSX[x+X](v,\infty,\cdot)f}_{p\left(\uS(v)\right)}^2 \\
 &\quad\leq 3\sup_{X\in\cX_T}\sum_{x\in\cG}\normb{\Psi_x \Psi_{x+X(v)}\PsiSX[x+X](v+1,\infty,\cdot) f - \Psi_{x}\PsiSX[x+X](v+1,\infty,\cdot) f}_{p\left(\uS(v)\right)}^2 \\
  &\quad+3\sup_{X\in\cX_T}\sum_{x\in\cG}\normb{\Psi_{x}\PsiSX[x+X](v+1,\infty,\cdot) f - \PsiSX[x+X](v+1,\infty,\cdot) f}_{p\left(\uS(v)\right)}^2 \\
  &\quad+3\sup_{X\in\cX_T}\sum_{x\in\cG}\normb{\Psi_{x+X(v)}\PsiSX[x+X](v+1,\infty,\cdot) f - \PsiSX[x+X](v+1,\infty,\cdot) f}_{p\left(\uS(v)\right)}^2 \\
  &\quad\leq 9 I(T,v+1,S). \qedhere
\end{align*}
\end{proof}
\begin{proof}[Proof of Lemma \ref{lemma:leaf}]
We start by studying
\begin{align}
\int \abs{ D_x \PsiSX(v,\infty,\oNt)f }\,d\oNt 
& = \int \abs{ D_x \Psi_{\oNSv} \PsiSX(v+1,\infty,\oNt)f }d\oNt. \label{eq:leaf-1}
\end{align}
Denote by $A$ the event that $\oNSv\cap ((x+\cN) \times [\uS(v),\oS(v)[)=\emptyset$, i.e., the Poisson point process of the graphical construction has no points in a neighbourhood of $x$ in the time interval $[\uS(v),\oS(v)[$. In that case $D_x$ and $\Psi_\oNSv$ commute, and
\begin{align}
\eqref{eq:leaf-1} &=  \int \ind_A \abs{\Psi_{\oNSv} D_x\PsiSX(v+1,\infty,\oNt)f} \,d\oNt \nonumber \\
& \quad+ \int \ind_{A^c}\abs{ D_x \Psi_{\oNSv} \PsiSX(v+1,\infty,\oNt)f} \,d\oNt \nonumber\\
&\leq \int \ind_A \Psi_{\oNSv} \abs{D_x\PsiSX(v+1,\infty,\oNt)f} \,d\oNt \label{eq:leaf-2}\\
& \quad+ \int \ind_{A^c}\abs{ D_x \Psi_{\oNSv} \PsiSX(v+1,\infty,\oNt)f} \,d\oNt. \label{eq:leaf-3}
\end{align}
By construction $\PsiSX(v+1,\infty,\oNt)$ does not depend on $\oN_{[0,\oS(v)[}$. Therefore we can split the integration into the time intervals $[0,\oS(v)]$ and $[\oS(v),t]$ and integrate separately:
\begin{align}
\eqref{eq:leaf-2} &= \left(\int  \ind_A \Psi_{\oNSv} d\oNSv\right)\left[ \int \abs{D_x\PsiSX(v+1,\infty,\oN_{[\uS(v),t]})f} \, d\oN_{[\uS(v),t]}\right]		\\
&\leq P_{S(v)} \left[\int  \abs{D_x\PsiSX(v+1,\infty,\oN_{[0,t]})f} \, d\oN_{[0),t]}\right].
\end{align}
Here we used the fact that the term in square brackets is a non-negative function, so that we can drop the indicator over $A$ to increase the value and then use Lemma \ref{lemma:graphical-construction-b} to write the first term as the semi-group. 

The estimate of \eqref{eq:leaf-3} is more involved. Let $M$ be $ Poisson\left(S(v)\abs{\cN}\right)$-distributed and $Q_1,Q_2,...$ be independent and uniformly distributed on $(x+\cN)\times [\uS(v),\oS(v)[$. Then $\oNSv\cap ((x+\cN)\times [\uS(v),\oS(v)[)$ has the the same distribution as $\bigcup_{k=1}^M \{Q_k\}$. The event $A^c$ corresponds to the event $\{M\geq 1\}$. Therefore, at least the point $Q_1$ is present. The number of remaining points is no longer Poisson distributed, but given by a Poisson distribution conditioned on at least one point, minus 1. However, this distribution has a density $h$ with respect to the Poisson distribution which satisfies $\norm{h}_\infty \leq 1+S(v)\abs{\cN}$. Therefore, we can replace the distribution of the remaining points by a Poisson point process and obtain for any positive function $g$:
\begin{align}
&\int \ind_{A^c} g(\oNt) \,d\oNt = (1-e^{-\abs{N}S(v)}) \int g(\oNt) \,d\left[\oNt\;\middle|\;A^c\right]	\\
&\leq \frac{(1-e^{-\abs{N}S(v)})}{S(v)\abs{\cN}} \int_{\uS(v)}^{\oS(v)}\sum_{y\in\cN} \int g\left(\oNt\cup\{(x+y,r)\}\right)h\left(\oNt\cup\{(x+y,r)\}\right) \,d\oNt dr	\\
&\leq 1\cdot (1+S(v)\abs{\cN}) \int_0^{S(v)}\sum_{y\in\cN} \int g\left(\oNt\cup\{(x+y,\uS(v)+r)\}\right) \,d\oNt dr. \label{eq:Poisson-replacement}
\end{align}
We now consider the function $g(\oNt)=\abs{D_x \Psi_\oNSv \PsiSX(v+1,\infty,\oNt)f}$. Remember the definition of $T'_v$, $S'_{r,v}$ and define $X'_y \in \cX_{T'_v}$ by $X'_y(v')=y$ and $X'_y(w')=X(w)$ for $w\in T$. Then we have for almost every $\oNt$ 
\begin{align}
g(\oNt \cup \{(x+y,\uS(v)+r)\} ) &= \abs{D_x \Psi_{\oNSv\cup\{(x+y,\uS(v)+r)\}} \PsiSX(v+1,\infty,\oNt \})f} \\
&= \abs{D_x \Psi_{\oN_{[\uS(v),\uS(v)+r[}}\Psi_{x+y}\Psi_{\oN_{[\uS(v)+r,\oS(v)[}} \PsiSX(v+1,\infty,\oNt)f}	\\
&= \abs{D_x \Psi_{S'_{r,v}, X'_y}(v'-1,\infty,\oNt)f}
\end{align}
Using the above with \eqref{eq:Poisson-replacement} gives us
\begin{align}
\eqref{eq:leaf-3} &\leq (1+S(v)\abs{\cN})\int_{0}^{S(v)}\sum_{y\in\cN}\int \abs{D_x \Psi_{S'_{r,v}, X'_y}(v'-1,\infty,\oNt)f}d\oNt dr. \label{eq:leaf-4}
\end{align}

Using all the previous estimates,
\begin{align}
 &I(T,v,S) \leq \sup_{X\in\cX_T}\sum_{x\in\cG} \Big( \normb{P_{S(v)}\abs{D_x \PsiSX[x+X](v+1,\infty, \cdot)f}}_{p\left(\uS(v)\right)} \nonumber \\
 &\quad+ (1+S(v)\abs{\cN})\int_0^{S(v)}\sum_{y\in\cN}\normb{D_x \Psi_{S'_{r,v},x+X'_y}(v'-1,\infty,\cdot)f }_{p\left(\uS(v)\right)} \,dr\Big)^2 \nonumber \\
 &\leq 2 \sup_{X\in\cX_T}\sum_{x\in\cG} \normb{P_{S(v)}\abs{D_x \PsiSX[x+X](v+1,\infty,\cdot)f}}_{p\left(\uS(v)\right)}^2 \label{eq:leaf-5}\\
 &\quad+ 2 \abs{\cN} S(v)(1+S(v)\abs{\cN})^2 \intc{0}{S(v)}\sum_{y\in\cN} \sup_{X\in\cX_T}\sum_{x\in\cG} \normb{D_x \Psi_{S'_{r,v},x+X'_y}(v'-1,\infty,\cdot)f}_{{p\left(\uS(v)\right)}}^2 \,dr .\label{eq:leaf-6}
\end{align}
Using hypercontractivity \eqref{eq:hypercontractivity}, 
\[ \eqref{eq:leaf-5} \leq 2 \sup_{X\in\cX_T}\sum_{x\in\cG} \normb{D_x \PsiSX[x+X](v+1,\infty,\cdot)f }_{p(\oS(v))}^2 = 2 I(T,v+1,S). \]
Since $\uS(v)=\underline{S'_{r,v}}(v'-1)$, and taking the supremum over $\cX_{T'_v}$,
\begin{align*}
 \eqref{eq:leaf-6} 
 &\leq 2 \abs{\cN}^2 S(v)(1+S(v)\abs{\cN})^2 \intc{0}{S(v)} \sup_{X'\in\cX_{T'_v}}\sum_{x\in\cG} \normb{D_x \Psi_{S'_{r,v},x+X'}(v'-1,\infty,\cdot)f}_{{p\left(\underline{S'_{r,v}}(v'-1)\right)}}^2 \,dr	\\
 &= 2 \abs{\cN}^2 S(v) (1+S(v)\abs{\cN})^2\int_0^{S(v)} I(T'_v,v'-1,S'_{r,v}) \,dr. \qedhere
\end{align*}
\end{proof}

\section{The measures \texorpdfstring{$m_{T,t}$}{m} and the iteration procedure}\label{section:full-iteration}
When looking at the error term introduced in Lemma \ref{lemma:leaf} we see that there is some integration happening. When applying Lemma \ref{lemma:leaf} (in the form of Corollary \ref{cor:leaf}) to the integrand of the error term, we get another integral. All together, to describe the total influence of an error term we will have to integrate with respect to some measure which depends on the tree structure and the time interval $[0,t]$. We call this measure $m_{T,t}$. 

Before defining the measure $m_{T,t}$ we have to introduce a bit of notation. Given a tree $T$, $T\neq \{0\}$, we denote by $T_L$ the \emph{left subtree} of $T$, which is the unique tree $T_L$ for which a map $\iota_L : T_L \to T$ with the following properties exists:
\begin{itemize}
 \item $\iota_L$ is injective and order preserving;
 \item $\iota_L(T_L) = \{ v\in T : v< 0 \}$;
 \item $\iota_L(0) = l(0)$.
\end{itemize}
Analogously we define $T_R$, the \emph{right subtree} of $T$. 

Fix a tree $T\neq \{0\}$ and $t\geq 0$. For $0\leq s \leq t$, choose $S_L\in \cS_{T_L, s}$ and $S_R\in\cS_{T_R,t-s}$. From the partitions $S_L$ and $S_R$ on the subtrees we define a corresponding $T$-partition of $[0,t]$,
\begin{align}
 (S_L, t, S_R)(v) := \begin{cases}
                      S_L(\iota_L^{-1}(v)),\quad&v<0;\\
                      t, &v=0;\\
                      S_R(\iota_R^{-1}(v)),\quad&v>0;
                     \end{cases}
\end{align}

For $T\in \cT_n$ and $t\geq 0$ we define the measure $m_{T,t}$ on $\cS_{T,t}$ recursively. If $T=\{0\}$, $m_{T,t} := \delta_{0\mapsto t}$, the Dirac measure on the single element of $\cS_{\{0\},t}$. If $T\neq \{0\}$,
\begin{align}
 m_{T,t} := t\int_0^t\int_{\cS_{T_L,s}} \int_{\cS_{T_R,t-s}} \delta_{(S_L, t, S_R)} \,m_{T_R,t-s}(dS_R)\,m_{T_L,s}(dS_L)\,ds.		\label{eq:m-definition}
\end{align}
The structure of $m_{T,t}$ is chosen to mirror the structure of the second right hand term in Lemma \ref{lemma:leaf} (for $t$ small). 
The recursive nature makes both terms somewhat difficult to understand. However, we obtain a reasonable estimate on the mass of $m_{T,t}$.
\begin{lemma}\label{lemma:m-mass}
For $T\in\cT_n$, $t\geq 0$,
\[ \intc{\cS_{T,t}}{} 1 \,m_{T,t}(dS) \leq \frac{t^{2n-2}}{(2n-3)!!}, \]
where $(2k-1)!! = \prod_{i=1}^k(2i-1)$ is the double factorial, with the convention of $(-1)!!=1$.
\end{lemma}
\begin{proof}
 The proof will be by induction over $n$. For $n=1$ both sides of the statement equal 1. For the induction step we use the recursive definition of $m_{T,t}$. Assume that the statement holds for all trees in $\cT_k, 1\leq k\leq n$ and all $t\geq 0$. For any $T\in \cT_{n+1}$ and any $t\geq 0$,
 \begin{align}
  \int_{\cS_{T,t}} 1 \,m_{T,t}(dS) &= t\int_0^t\int_{\cS_{T_L,s}} \int_{\cS_{T_R,t-s}} 1 \,m_{T_R,t-s}(dS_R)\,m_{T_L,s}(dS_L)\,ds \nonumber\\
  &\leq t\int_0^t\frac{s^{2n_L-2}}{(2n_L-3)!!}\frac{(t-s)^{2n_R-2}}{(2n_R-3)!!}\,ds, \label{eq:mass-1}
 \end{align}
where $T_L\in \cT_{n_L}$ and $T_R\in \cT_{n_R}$. Using partial integration,
\begin{align*}
 \eqref{eq:mass-1} = \frac{(2n_L-2)!(2n_R-2)!}{(2n_L-3)!!(2n_R-3)!!(2n_L+2n_R-3)!} t^{2n_R+2n_L-2} 
\end{align*}
A simple calculation shows that $\frac{(2n_L-2)!(2n_R-2)!}{(2n_L-3)!!(2n_R-3)!!(2n_L+2n_R-3)!} \leq \frac{1}{(2n_L+2n_r-3)!!}$, and since $n_L+n_R=n+1$ this completes the proof.
\end{proof}

To facilitate the recursive argument which will ultimately lead to a proof of Theorem \ref{thm:nice} we need to look at the vertices of a tree in more detail. First, we denote by
\begin{align}
 \#v_- := \abs{\left\{ w \in \partial T : w<v \right\}},\quad T\in\cT, v\in T,
\end{align}
the number of leaves before a vertex $v$. Next we define 
\begin{align}
 B(T)=\max\{v\in \oT : v-1,v+1 \in \partial T\},\quad T\in\cT,
\end{align}
with the convention that the maximum of the empty set is $-\infty$. In other words, $B(T)$ is the last simple branching point in $T$, with simple referring to the fact that both its children are leaves.
\begin{lemma}\label{lemma:n-G}
For a tree $T$ and $T'_v$ as in Lemma \ref{lemma:leaf}, we have the following:
 \begin{enumerate}
  \item Let $v\in\partial T$ be a leaf which satisfies $v\geq B(T)-1$. Then in the expanded tree $T'_v$, $v'=B(T'_v)$. In other words, the embedding $v'$ of $v$ into $T'_v$ is the last simple branching point in $T'_v$. Additionally, the number of leaves before $v$ satisfies $\#v_- = \#B(T'_v)_--1$.
  \item The set $\cT_{n+1}$ of full binary trees with exactly $n+1$ leaves can be obtained as the disjoint union of expansions of the trees in $\cT_n$: $\cT_{n+1} = \mathring{\bigcup}_{T\in\cT_n}\left\{ T'_v : v\in\partial T, v\geq B(T)-1\right\}$.
 \end{enumerate}
\end{lemma}
\begin{proof}
 If $v$ is not a child of $B(T)$, then $v>B(T)$ and $v'=B(T'_v)$. If $v$ is a child of $B(T)$, then $B(T)'$ is no longer a simple branching point, and hence $v'=B(T'_v)$. From this and $\#v'_-=\#v_-+1$ follows $\#v_- = \#B(T'_v)_--1$.
 
 For a tree $T\in \cT_{n+1}$, let $T^*\in\cT_n$ be the tree obtained by removing the two child leaves of $B(T)$. By slight abuse of notation we also write $B(T)$ for the leaf in $T^*$ so that $(T^*)'_{B(T)}=T$. By construction of $T^*$, $B(T^*)< B(T)$ or $B(T)$ is a child of $B(T^*)$. Hence
 $ \cT_{n+1} \subset \bigcup_{T\in\cT_n}\left\{ T'_v : v\in\partial T, v\geq B(T)-1\right\}$.
 The other inclusion is obvious. To show that the union is disjoined, we observe that $(T'_v)^*=T$ for any $v\in\partial T, v\geq B(T)-1$.
\end{proof}

We are now in state to use Lemmas \ref{lemma:interior} and \ref{lemma:leaf} iteratively, providing us with a first commutation property between $D_x$ and $P_t$. This iteration procedure is captured in the following theorem.
\begin{theorem}\label{thm:power}
Suppose $t\leq 1$. For any $N\geq0$,
\begin{align*}
\sum_{x \in \cG}\norm{D_x P_t f}^2_2 
&\leq \left( \sum_{n=1}^N 2\cdot18^{n-1}(2\cdot \abs{\cN}^2(1+\abs{\cN})^2)^{n-1} \frac{\abs{\cT_n}}{(2n-3)!!}t^{2n-2} \right) \sum_{x\in\cG}\norm{D_x f}^2_{p(t)}\\
&+(2\cdot\abs{\cN}^2(1+\abs{\cN})^2)^N \sum_{\mathclap{T\in\cT_{N+1}}} 18^{\#(B(T)-1)_-}\int_{\cS_{T,t}} I(T,B(T)-1,S) \,m_{T,t}(dS).
\end{align*}
\end{theorem}
\begin{proof}
The proof will be by induction. For $N=0$, we observe that
\begin{align*}
&\sum_{T\in\cT_{1}} 18^{\#(B(T)-1)_-}\int_{\cS_{T,t}} I(T,B(T)-1,S) \,m_{T,t}(dS) \\
&\quad= I(\{0\},-\infty, 0\mapsto t) \geq \sum_{x\in\cG}\norm{D_x P_t f}_2^2,
\end{align*}
where we used \eqref{eq:I-at-0} for the last step.
To show the induction step, it suffices to prove that
\begin{align*}
 & \sum_{T\in\cT_{N}} 18^{\#(B(T)-1)_-}\int_{\cS_{T,t}} I(T,B(T)-1,S) \,m_{T,t}(dS) \\
 &\quad\leq  2\cdot18^{N-1} \frac{\abs{\cT_N}}{(2N-3)!!}t^{2N-2} \cdot\sum_{x\in\cG}\norm{D_x f}^2_{p(t)}  \\
 &\quad+2\cdot\abs{\cN}^2 (1+\abs{\cN})^2 \sum_{T\in\cT_{N+1}} 18^{\#(B(T)-1)_-}\int_{\cS_{T,t}} I(T,B(T)-1,S) \,m_{T,t}(dS).
\end{align*}
Fix $T\in\cT_N$ and let $v\in\partial T$. Now we apply Corollary \ref{cor:leaf} to $I(T,v,S)$. If $v+1=\infty$, then
\begin{align*}
I(T,v,S) &\leq 2\sum_{x\in\cG}\norm{D_x f}_{p(t)}^2 + 2\cdot \abs{\cN}^2(1+\abs{\cN})^2 S(v)\int_0^{S(v)} I(T'_v, v'-1, S'_{r,v})dr.
\end{align*}
Otherwise, $v+2<\infty$ as well, and using Lemma \ref{lemma:interior} in addition to Corollary \ref{cor:leaf} yields
\begin{align*}
I(T,v,S) &\leq 2I(T,v+1,S) + 2\cdot \abs{\cN}^2(1+\abs{\cN})^2 S(v)\int_0^{S(v)} I(T'_v, v'-1, S'_{r,v})dr \\
&\leq 18 I(T,v+2,S) + 2\cdot \abs{\cN}^2(1+\abs{\cN})^2 S(v)\int_0^{S(v)} I(T'_v, v'-1, S'_{r,v})dr.
\end{align*}
As $v+2$ is again a leaf, we can iterate. Therefore,
\begin{align*}
I(T,B(T)-1,S) &\leq \sum_{\mathclap{\substack{v\in\partial T \\ v\geq B(T)-1}}} 18^{\#v_- - \#(B(T)-1)_-} 2\cdot \abs{\cN}^2(1+\abs{\cN})^2 S(v)\intc{0}{S(v)} I(T'_v, v'-1, S'_{r,v})\,dr \\
&+2 \cdot 18^{N-1-\#(B(T)-1)_-}\sum_{x\in\cG}\norm{D_x f}^2_{p(t)}.
\end{align*}
Integrating with respect to $m_{T,t}$ and multiplying with $18^{\#(B(T)-1)_-}$, we get
\begin{align*}
&18^{\#(B(T)-1)_-} \int_{\cS_{T,t}} I(T,v,S)\,m_{T,t}(dS) \\
&\quad\leq \sum_{\substack{v\in\partial T \\ v\geq B(T)-1}} 18^{\#v_-} 2\cdot \abs{\cN}^2 (1+\abs{\cN})^2 \int_{\cS_{T,t}}S(v)\int_0^{S(v)} I(T'_v, v'-1, S'_{r,v})\,dr \, m_{T,t}(dS) \\
&\quad+2 \cdot 18^{N-1}\int_{\cS_{T,t}} 1 \,m_{T,t}(dS) \sum_{x\in\cG}\norm{D_x f}^2_{p(t)}.
\end{align*}
Note that by Lemma \ref{lemma:m-mass}, 
\[ \intc{\cS_{T,t}}{} 1 \,m_{T,t}(dS) \leq \frac{t^{2N-2}}{(2N-3)!!} . \]
By construction of the measures $m_{T,t}$ and using the Lemma \ref{lemma:n-G},
\begin{align*}
 &\int_{\cS_{T,t}}S(v)\int_0^{S(v)} I(T'_v, v'-1, S'_{r,v})\,dr \, m_{T,t}(dS) \\
 &\quad= \int_{\cS_{T'_v,t}} I(T'_v, B(T'_v)-1, S') \, m_{T'_v,t}(dS')
\end{align*}
Summing over $T\in\cT_N$ completes the proof by using Lemma \ref{lemma:n-G} again.
\end{proof}
\begin{proof}[Proof of Theorem \ref{thm:nice}]
The proof consists of two parts.  
In a first step, we will show that for $t\leq 1$
\begin{align}\label{eq:nice-1}
 \sum_{x\in\cG} \norm{D_x P_t f}^2_2 \leq 2e^{\lambda t^2}\sum_{x\in\cG}\norm{D_x f}^2_{p(t)},
\end{align}
with $\lambda=72\cdot\abs{\cN}^2(1+\abs{\cN})^2$.

It is known from combinatorics that $\abs{\cT_{n+1}}=C_n$, where $C_n= \frac{(2n)!}{(n+1)!n!}$ is the $n$th Catalan number \cite{STANLEY:99}. Therewith 
\[ \frac{\abs{\cT_{n+1}}}{(2n-1)!!} = \frac{2^n}{(n+1)!}, \]
and by Theorem \ref{thm:power},
\begin{align*}
 \sum_{x\in\cG}\norm{D_x P_t f}_2^2 \leq 2\sum_{n=0}^\infty \frac{1}{(n+1)!}\left(72\abs{\cN}^2(1+\abs{\cN})^2 t^2\right)^n \leq 2 e^{72\abs{\cN}^2(1+\abs{\cN})^2 t^2},
\end{align*}
showing \eqref{eq:nice-1}. To obtain the claim \eqref{eq:nice-thm} from \eqref{eq:nice-1} we first note that without changes in the proofs except for notation
\begin{align}
 \sum_{x\in\cG} \norm{D_x P_t f}^2_{p(s)} \leq 2e^{\lambda t^2}\sum_{x\in\cG}\norm{D_x f}^2_{p(s+t)}	\label{eq:selfimprove}
\end{align}
holds as well for any $s\geq 0$. When we write $P_t= (P_{t/\lceil t \rceil})^{\lceil t \rceil}$, and apply \eqref{eq:selfimprove} $\lceil t \rceil$ times, we get
\begin{align}
 \sum_{x\in\cG} \norm{D_x P_t f}^2_2 \leq 2^{\lceil t \rceil} e^{\lambda \left(\frac{t}{\lceil t \rceil}\right)^2}\sum_{x\in\cG}\norm{D_x f}^2_{p(t)}.
\end{align}
Using $\lceil t \rceil \leq t+1$ and $t/\lceil t \rceil \leq 1$ 
 yields \eqref{eq:nice-thm} with $\widetilde C=2e^{\lambda}$ to show the theorem. 
\end{proof}

\bibliographystyle{plain}
\bibliography{BibCollection}

\begin{thebibliography}{10}

\bibitem{BENJAMINI:KALAI:SCHRAMM:03}
Itai Benjamini, Gil Kalai, and Oded Schramm.
\newblock First passage percolation has sublinear distance variance.
\newblock {\em Ann. Probab.}, 31(4):1970--1978, 2003.

\bibitem{VANDENBERG:KISS:12}
Jacob van~den Berg and Demeter Kiss.
\newblock Sublinearity of the travel-time variance for dependent first-passage
  percolation.
\newblock {\em Ann. Probab.}, 40(2):743--764, 2012.

\bibitem{BOBKOV:GOETZE:99}
S.G Bobkov and F~G{\"o}tze.
\newblock Exponential integrability and transportation cost related to
  logarithmic sobolev inequalities.
\newblock {\em Journal of Functional Analysis}, 163(1):1 -- 28, 1999.

\bibitem{BOBKOV:HOUDRE:99}
S.G. Bobkov and C.~Houdr{\'e}.
\newblock A converse gaussian poincaré-type inequality for convex functions.
\newblock {\em Statistics \& Probability Letters}, 44(3):281 -- 290, 1999.

\bibitem{BOURGAIN:KAHN:KALAI:92}
Jean Bourgain, Jeff Kahn, Gil Kalai, Yitzhak Katznelson, and Nathan Linial.
\newblock The influence of variables in product spaces.
\newblock {\em Israel Journal of Mathematics}, 77(1-2):55--64, 1992.

\bibitem{CORDERO-ERAUSQUIN:LEDOUX:12}
Dario Cordero-erausquin and Michel Ledoux.
\newblock {Hypercontractive Measures, Talagrand’s Inequality, and
  Influences}.
\newblock {\em Geometric Aspects of Functional Analysis}, 2050:169--189, 2012.

\bibitem{FRIEDGUT:99}
Ehud Friedgut.
\newblock Sharp thresholds of graph properties, and the {$k$}-sat problem.
\newblock {\em J. Amer. Math. Soc.}, 12(4):1017--1054, 1999.
\newblock With an appendix by Jean Bourgain.

\bibitem{GUIONNET:ZEGARLINSKI:03}
A.~Guionnet and B.~Zegarlinksi.
\newblock Lectures on logarithmic sobolev inequalities.
\newblock In Jacques Az{\'e}ma, Michel {\'E}mery, Michel Ledoux, and Marc Yor,
  editors, {\em S{\'e}minaire de Probabilit{\'e}s XXXVI}, volume 1801 of {\em
  Lecture Notes in Mathematics}, pages 1--134. Springer Berlin Heidelberg,
  2003.

\bibitem{HOLLEY:STROOK:87}
Richard Holley and Daniel Stroock.
\newblock Logarithmic sobolev inequalities and stochastic ising models.
\newblock {\em Journal of Statistical Physics}, 46:1159--1194, 1987.
\newblock 10.1007/BF01011161.

\bibitem{KAHN:KALAI:LINIAL:88}
Jeff Kahn, G.~Kalai, and Nathan Linial.
\newblock The influence of variables on boolean functions.
\newblock In {\em Foundations of Computer Science, 1988., 29th Annual Symposium
  on}, pages 68--80, 1988.

\bibitem{KELLER:MOSSEL:SEN:12}
Nathan Keller, Elchanan Mossel, and Arnab Sen.
\newblock Geometric influences.
\newblock {\em The Annals of Probability}, 40(3):1135--1166, 2012.

\bibitem{KRASNOSELSKII:RUTITSKII:61}
M.A. Krasnoselʹski{\u\i} and I.A.B. Rutitski{\u\i}.
\newblock {\em Convex functions and Orlicz spaces}.
\newblock P. Noordhoff, 1961.

\bibitem{LEONARD:07}
Christian L{\'e}onard.
\newblock Orlicz spaces.
\newblock http://www.cmap.polytechnique.fr/{$\sim$}leonard/papers/orlicz.pdf,
  2007.

\bibitem{MARTINELLI:OLIVIERI:94a}
F.~Martinelli and E.~Olivieri.
\newblock Approach to equilibrium of glauber dynamics in the one phase region.
  {I}. {T}he attractive case.
\newblock {\em Communications in Mathematical Physics}, 161(3):447--486, 1994.

\bibitem{MARTINELLI:OLIVIERI:94}
F.~Martinelli and E.~Olivieri.
\newblock Approach to equilibrium of {G}lauber dynamics in the one phase
  region. {II}. {T}he general case.
\newblock {\em Comm. Math. Phys.}, 161(3):487--514, 1994.

\bibitem{STANLEY:99}
Richard~P. Stanley.
\newblock {\em Enumerative Combinatorics, Volume 2}.
\newblock Cambridge University Press, 1999.

\bibitem{TALAGRAND:94}
Michel Talagrand.
\newblock {On Russo's Approximate Zero-One Law}.
\newblock {\em The Annals of Probability}, 22(3):1576--1587, July 1994.

\end{thebibliography}

\end{document}